\documentclass[11pt,twoside]{amsart}
\usepackage{amsfonts}
\usepackage{amsmath}
\usepackage{amssymb}
\usepackage[latin1]{inputenc}
\usepackage{graphicx}
\usepackage{mathrsfs}
\usepackage{cancel}

\newtheorem{theo}{Theorem}
\newtheorem{cor}{Corollary}
\newtheorem{lem}{Lemma}
\newtheorem{prop}{Proposition}
\theoremstyle{definition}

\theoremstyle{remark}
\newtheorem{rem}{\bf Remark\/}

\newtheorem{exple}{\bf Example\/}
\numberwithin{equation}{section}

\def\C{{\mathbb{C}}}

\newcommand{\ds}{\displaystyle}
\newcommand{\w}{\wedge}
\newcommand{\PSH}{\mathrm{PSH}}

\title[On the Lelong-Demailly numbers of PSH currents]{On the Lelong-Demailly numbers of plurisubharmonic currents}
\author[N. Ghiloufi]{Noureddine Ghiloufi}
\email{noureddine.ghiloufi@fsg.rnu.tn}
\address{Department of Mathematics\\ Faculty of sciences of Gab\`es \\ University of Gab\`es \\ 6072 Gab\`es Tunisia.}
\begin{document}
\maketitle

\begin{abstract}
    In this note we study the existence of the Lelong-Demailly number of a negative plurisubharmonic current with respect to a positive plurisubharmonic function on an open subset of $\C^n$. Then we establish some estimates of the Lelong-Demailly numbers of positive or negative plurisubharmonic currents.\\

    \textbf{Sur les Nombres de Lelong-Demailly des courants plurisous\-harmoniques.}\\
    \textsc{R\'esum\'e.} Dans cette note, on \'etudie l'existence du nombre de Lelong-Demailly d'un courant n\'egatif plurisous\-harmonique relativement \`a une fonction positive plurisousharmonique sur un ouvert de $\C^n$ puis on donne quelques estimations des nombres de Lelong-Demailly des courants positifs ou n\'egatifs plurisousharmoniques.
\end{abstract}
\vskip0.5cm
\hrule
\vskip0.5cm
\section*{Version fran\c{c}aise abr\'eg\'ee}
    L'existence des nombres de Lelong des courants positifs a \'et\'e r\'esolu par P.~Lelong dans les ann\'ees 1950 pour le cas des courants \emph{ferm\'es}, puis ce r\'esultat a \'et\'e \'etendu par Skoda au cas des courants \emph{positifs plurisousharmoniques}. En revanche, il existe des courants \emph{n\'egatifs plurisousharmoniques} qui n'admettent pas de nombres de Lelong, on peut voir par exemple que $\log(|z_2|^2)[z_1=0]$ est un exemple de courant  n\'egatif plurisousharmonique de bidimension $(1,1)$ sur la boule unit\'e de $\C^2$ qui n'admet pas de nombre de Lelong en 0.\\
    Le principal objectif de cette note est de traiter le probl\`eme de l'existence des nombres de Lelong g\'en\'eralis\'es introduits par Demailly \cite{De}, d'un courant n\'egatif plurisousharmonique $T$ de bidimension $(p,p)$ sur un ouvert $\Omega$ de $\C^n$; pour cela on note $\PSH(T,\Omega)$ l'ensemble des fonctions  $\varphi$ positives plurisousharmoniques semi-exhaustives dont le logarithme $\log\varphi$ est plurisousharmonique sur $\Omega$, et telles que le produit ext\'erieur $T\w(dd^c\varphi)^p$ soit bien d\'efini. Le nombre de Lelong-Demailly de $T$ relativement \`a un poids $\varphi$ tel que $\varphi\in\PSH(T,\Omega)$ est $\nu(T,\varphi):= \lim_{r\to0^+} \nu(T,\varphi,r)$, o\`u $\nu(T,\varphi,.)$ est la fonction d\'efinie  par
    $$\nu(T,\varphi,r):=\frac1{r^p}\int_{\{\varphi<r\}}T\w(dd^c\varphi)^p.$$
    Le r\'esultat principal de cette note est:\\

    \noindent\textbf{Th\'eor\`eme~\ref{theo1}.} \textit{Soient $T$ un courant n\'egatif plurisousharmonique de bidimension $(p,p)$ sur $\Omega$ et $\varphi\in\PSH(T,\Omega)$. Si la fonction $t\longmapsto \frac{\nu(dd^cT,\varphi,t)}{t}$ est int\'egrable au voisinage de $0$, alors le nombre de Lelong-Demailly $\nu(T,\varphi)$ du courant $T$ relativement \`a $\varphi$ existe.}
\section{Introduction}
    The existence of Lelong numbers of positive currents was proved by P.~Lelong in the 1950's in the  case of \emph{closed} currents, then this result was extended by H.~Skoda to the case of \emph{positive plurisubharmonic} currents. However, there are \emph{negative} plurisubharmonic  currents which do not admit Lelong numbers, for example $-(-\log(|z_2|^2))^\epsilon[z_1=0]$ is a negative plurisubharmonic current of bidimension $(1,1)$ on the unit ball of $\C^2$, which admits no Lelong number at 0 for all $0<\epsilon\leq1$.\\
    The main purpose of the first part of this note is to study the existence of generalized Lelong numbers, introduced by Demailly \cite{De}, in the case of  negative plurisubharmonic currents of bidimension $(p,p)$ on an open set $\Omega$ of $\C^n$. In the second part, we give some proprieties of the Lelong-Demailly numbers of positive or negative plurisubharmonic currents. In particular, we prove that the Lelong-Demailly numbers do not depend to the system of coordinates. To this aim, we consider a non-negative plurisubharmonic function $\varphi$ on $\Omega$ such that $\log\varphi$ is plurisubharmonic on $\{\varphi> 0\}$ and for every $r>0,\ r_2>r_1>0$, we set
    $$\begin{array}{l}
        \ds B_\varphi(r):=\{z\in\Omega;\ \varphi(z)<r\},\\
        \ds B_\varphi(r_1,r_2):=\{z\in\Omega;\ r_1\leq\varphi(z)<r_2\}=B_\varphi(r_2)\smallsetminus B_\varphi(r_1) \\
        \ds \beta_\varphi:=dd^c\varphi=\frac{i}{\pi}\partial\overline{\partial}\varphi,\quad \alpha_\varphi:=dd^c\log\varphi \hbox{ on } \{\varphi>0\}.
      \end{array}$$
    We assume that $\varphi$ is semi-exhaustive, i.e. there exists $R=R(\varphi)>0$ such that $B_\varphi(R)$ is relatively compact in $\Omega$. If $S$ is a positive (or negative) current of  bidimension $(p,p)$ on the set $\Omega$ then we denote $\PSH(S,\Omega)$ the set of non-negative semi-exhaustive plurisubharmonic functions  $\varphi$ such that $\log\varphi$ is also plurisubharmonic on $\Omega$  and the exterior product $S\w(dd^c\varphi)^p$ is well defined. Finally, we denote by $\mathcal I_S(\varphi):=\{k>0;\ \varphi^k\in \PSH(S,\Omega)\}$ for every $\varphi\in \PSH(S,\Omega)$; in particular, if $\varphi$ is $\mathcal C^2$ then $\mathcal I_S(\varphi)\supset[1,+\infty[$ for every current $S$.\\
    The Lelong-Demailly number of $S$ relatively to the weight $\varphi$ is $\nu(S,\varphi):= \lim_{r\to0^+} \nu(S,\varphi,r)$ where $\nu(S,\varphi,.)$ is the function defined on $]0,R(\varphi)[$ by
    $$\nu(S,\varphi,r):=\frac1{r^p}\int_{B_\varphi(r)}S\w\beta_\varphi^p.$$
    The classical Lelong number is given by the choice $\varphi(z)=\varphi_0(z):=|z|^2$.
    \begin{exple}\label{exple1}
        Let $S_\epsilon(z_1,z_2)=(|z_2|^{2\epsilon}-1)[z_1=0]$ where $\epsilon>0$. Then $S_\epsilon$ is a negative plurisubharmonic current of  bidimension $(1,1)$ on the unit ball $\mathbb B$ of $\C^2$, $\varphi_0\in \PSH(S_\epsilon,\mathbb B)$ and $\mathcal I_{S_\epsilon}(\varphi_0)=\mathcal I_{dd^cS_\epsilon}(\varphi_0)={}]0,+\infty[.$
    \end{exple}
    \begin{proof}
        A simple computation  proves that $$S_\epsilon\w dd^c(\varphi_0^k)=\frac{k^2}\pi (|z_2|^{2\epsilon}-1)|z_2|^{2(k-1)}idz_2\w d\overline{z}_2\w[z_1=0]$$ and $dd^cS_\epsilon=\frac{\epsilon^2}\pi|z_2|^{2(\epsilon-1)}idz_2\w d\overline{z}_2\w[z_1=0]$, hence $S_\epsilon\w dd^c(\varphi_0^k)$ is well  defined for all $k>0$. Therefore, $\mathcal I_{S_\epsilon}(\varphi_0)=\mathcal I_{dd^cS_\epsilon}(\varphi_0)={}]0,+\infty[$. One can check easily that
        $$\nu(S_\epsilon,\varphi_0^k,r)=2k^2\left(\frac{r^{\epsilon/k}}{\epsilon+k}-\frac1k\right)\quad \hbox{and}\quad \nu(dd^cS_\epsilon,\varphi_0,r)=2\epsilon r^{\epsilon},$$
        hence $\nu(S_\epsilon,\varphi_0^k)=-2k$.
    \end{proof}
\section{Main result}
    In the following, we will use a Lelong-Jensen formula proved by  Demailly \cite{De}. In \cite{To}, Toujani used an analogue of this formula to prove the existence of the directional Lelong-Demailly numbers of positive plurisubharmonic currents.
    \begin{lem}\label{lem1}(See \cite{De} or \cite{To})
        Let $S$ be a positive or negative plurisubharmonic current of bidimension $(p,p)$ on $\Omega$ and $\varphi\in \PSH(S,\Omega)$. Then, for all $0<r_1<r_2<R(\varphi)$,
        \begin{equation}\label{eq 1.1}
            \begin{array}{lcl}
                \nu(S,\varphi,r_2)-\nu(S,\varphi,r_1) & = &\ds \frac1{r_2^p}\int_{B_\varphi(r_2)} S\w\beta_\varphi^p -\frac1{r_1^p} \int_{B_\varphi(r_1)}S\w\beta_\varphi^p\\
                & = & \ds \int_{B_\varphi(r_1,r_2)}S\w\alpha_\varphi^p\\
                & & \ds +\int_{r_1}^{r_2}\left(\frac1{t^p} -\frac1{r_2^p} \right)dt\int_{B_\varphi(t)}dd^cS\w\beta_\varphi^{p-1}\\
                & &\ds +\left(\frac1{r_1^p}-\frac1{r_2^p}\right)\int_0^{r_1}dt\int_{B_\varphi(t)} dd^cS\w \beta_\varphi^{p-1}.
            \end{array}
        \end{equation}
    \end{lem}
    According to Lemma \ref{lem1}, if $S$ is  positive plurisubharmonic then $\nu(S,\varphi,.)$ is a non-negative increasing function on $]0,R(\varphi)[$, so  $\nu(S,\varphi):=\lim_{r\to0^+} \nu(S,\varphi,r)$  exists.\\
    It is well known that if $S$ is a positive plurisubharmonic current, then for every $\varphi\in \PSH(S,\Omega)$ the function $t\longmapsto\frac{\nu(dd^cS,\varphi,t)}t$ is integrable in the neighborhood of 0. Throughout this note, for every negative plurisubharmonic  current $T$ on $\Omega$, we say that a function $\varphi\in \PSH(T,\Omega)$ satisfies condition $(C)$ if the function $t\longmapsto \frac{\nu(dd^cT,\varphi,t)}{t}$ is integrable on a neighborhood of 0. In particular, if $\varphi$ satisfies condition $(C)$, we must have $\nu(dd^cT,\varphi)=0$.\\
    Now we state the main result concerning the case of negative plurisubharmonic currents.
    \begin{theo}\label{theo1}
        Let $T$ be a negative plurisubharmonic current of  bidimension $(p,p)$ on $\Omega$ and $\varphi\in \PSH(T,\Omega)$ satisfying condition $(C)$. Then, the Lelong-Demailly number $\nu(T,\varphi)$ of the current $T$ relatively to $\varphi$ exists.
    \end{theo}
    \begin{proof}
        For every $0<r<R(\varphi)$, we set
        $$f(r)=\frac1{r^p}\int_{B_\varphi(r)}T\w\beta_\varphi^p+\frac1{r^p}\int_0^r dt\int_{B_\varphi(t)}dd^cT\w\beta_\varphi^{p-1} -\int_0^r\frac{dt}{t^p}\int_{B_\varphi(t)}dd^cT\w\beta_\varphi^{p-1}.$$
        Thanks to condition $(C)$, the function $f$ is well defined and non-positive on $]0,R(\varphi)[$. Indeed,
        $$f(r) =\nu(T,\varphi,r)+\int_0^r\left(\frac{t^p}{r^p}-1\right)\frac{\nu(dd^cT,\varphi,t)}{t}dt\leq0$$
        because the  function $\nu(dd^cT,\varphi,.)$ is non-negative on $]0,R(\varphi)[$.\\
        For $0<r_1<r_2<R(\varphi)$ we set $A(r_1,r_2):=f(r_2)-f(r_1)$. The current $T$ is negative plurisubharmonic, thus lemma \ref{lem1} gives
        \begin{equation}\label{eq 1.2}
            \begin{array}{lcl}
                A(r_1,r_2) & = &\ds\nu(T,\varphi,r_2)- \nu(T,\varphi,r_1)+ \frac1{r_2^p}\int_0^{r_2}t^{p-1} \nu(dd^cT,\varphi,t)dt\\
                & & \ds-\frac1{r_1^p}\int_0^{r_1} t^{p-1}\nu(dd^cT,\varphi,t)dt -\int_{r_1}^{r_2}\frac{\nu(dd^cT,\varphi,t)}{t}dt\\
                &=&\ds \int_{B_\varphi(r_1,r_2)}T\w\alpha_\varphi^p\leq 0.
            \end{array}
        \end{equation}
        Therefore, $f$ is a non-positive decreasing function on $]0,R(\varphi)[$, and this implies the existence of the  limit $\varrho:=\lim_{r\to0^+} f(r)$. The hypothesis of integrability of $\nu(dd^cT,\varphi,t)/t$ and the fact that $(t^p/r^p-1)$ is uniformly  bounded give $$\ds\lim_{r\to0^+}\int_0^r\left(\frac{t^p}{r^p}-1\right)\frac{\nu(dd^cT,\varphi,t)}{t}dt=0.$$
        Therefore, $\varrho=\lim_{r\to0^+}f(r)=\lim_{r\to0^+}\nu(T,\varphi,r)=\nu(T,\varphi)$.
    \end{proof}
    \begin{rem}
        If $T$ is a positive (resp. negative) plurisubharmonic current of bidimension $(p,p)$ on $\Omega$ and $\varphi\in\PSH(T,\Omega)$ (resp. satisfying condition $(C)$) then the extension $\widetilde{T\w\alpha_\varphi^p}$ of $T\w\alpha_\varphi^p$ by 0 over the compact set $\{\varphi=0\}$ exists and we have
        $$\int_{B_\varphi(r)}\widetilde{T\w\alpha_\varphi^p}=f(r)-\nu(T,\varphi).$$
    \end{rem}
    Indeed, it suffice to use  the lemma \ref{lem1} to prove that $\int_{B_\varphi(\epsilon,r)}T\w\alpha_\varphi^p$ is uniformly bounded  and then tend $\epsilon$ to 0 to prove the equality of the remark.\\
    A natural question arises: does the existence of the Lelong-Demailly number of the negative plurisub\-harmonic current $T$ implies the existence of $\widetilde{T\w\alpha_\varphi^p}$? so condition $(C)$ will be necessary in Theorem \ref{theo1}.\\

    In the following proposition, we give some properties of Lelong-Demailly numbers in both cases of negative or positive plurisubharmonic currents.
    \begin{prop}\label{pro1}
        Let $T$ be a positive or negative plurisubharmonic current of bidimension $(p,p)$ on $\Omega$ and $\varphi\in \PSH(T,\Omega)$. Then, for every $k\in\mathcal I_T(\varphi)$ and every $r\in{}]0,R(\varphi)[$, we have 
        \begin{equation}\label{eq 1.3}
            \nu(T,\varphi^k,r^k)=k^p\left[\nu(T,\varphi,r)+\int_0^r\frac{\nu(dd^cT,\varphi,t)}t \left(\frac{t^p}{r^p} -\frac{t^{kp}}{r^{kp}}\right)dt\right].
        \end{equation}
        In particular, if $T$ is negative, then $\nu(T,\varphi)$ exists if and only if  for all $k\in\mathcal I_T(\varphi)$,  $\nu(T,\varphi^k)$ exists.\\
        Furthermore, in both cases $($with the assumption that  $\nu(T,\varphi)$ exists if $T$ is negative$)$ we have $$\nu(T,\varphi^k)=k^p\nu(T,\varphi).$$
    \end{prop}
    The previous equality is due to Demailly in the case of closed positive currents.
    \begin{proof}
        Let $\epsilon>0$ be sufficiently small. By replacing $\varphi$ with $\varphi_\epsilon=\varphi+\epsilon$ and $\psi:=\varphi^k$ with $\psi_\epsilon=\varphi_\epsilon^k$, when $T$ is negative (respectively positive), Equality (\ref{eq 1.2}) (resp.\ Equality (\ref{eq 1.1})) gives for $0<r_1< \epsilon<r<R(\varphi)$
        \begin{equation}\label{eq 1.4}
            \begin{array}{lcl}
              \ds\nu(T,\varphi_\epsilon,r)  & = & \ds \int_{B_{\varphi_\epsilon}(\epsilon,r)}T\w\alpha_{\varphi_\epsilon}^p -\frac1{r^p}\int_\epsilon^rt^{p-1}\nu(dd^cT,\varphi_\epsilon,t)dt  \\
              & &\ds \hfill +\int_\epsilon^r\frac{\nu(dd^cT,\varphi_\epsilon,t)}t dt
            \end{array}
        \end{equation}
        and
        \begin{equation}\label{eq 1.5}
        \begin{array}{lcl}
            \nu(T,\psi_\epsilon,r^k) & =& \ds \int_{B_{\psi_\epsilon}(\epsilon^k,r^k)}T\w\alpha_{\psi_\epsilon}^p -\frac1{r^{kp}}\int_{\epsilon^k}^{r^k}t^{p-1}\nu(dd^cT,\psi_\epsilon,t)dt\\ 
            & &\ds\hfill+\int_{\epsilon^k}^{r^k}\frac{\nu(dd^cT,\psi_\epsilon,t)}t dt\\
            & = & \ds \int_{B_{\psi_\epsilon}(\epsilon^k,r^k)}T\w\alpha_{\psi_\epsilon}^p - \frac k{r^{kp}}\int_\epsilon^r s^{kp-1}\nu(dd^cT,\psi_\epsilon,s^k)ds\\ 
            & &\ds\hfill+k\int_\epsilon^r\frac{\nu(dd^cT,\psi_\epsilon,s^k)}s ds\\
            & =& \ds k^p\int_{B_{\varphi_\epsilon}(\epsilon,r)}T\w\alpha_{\varphi_\epsilon}^p -\frac{k^p}{r^{kp}}\int_\epsilon^r s^{kp-1}\nu(dd^cT,\varphi_\epsilon,s)ds \\ 
            & &\ds\hfill +k^p\int_\epsilon^r\frac{\nu(dd^cT,\varphi_\epsilon,s)}s ds\\
            & = &\ds k^p\left(\nu(T,\varphi_\epsilon,r)+\frac1{r^p}\int_\epsilon^rt^{p-1} \nu(dd^cT,\varphi_\epsilon,t)dt \right)\\
            & & \hfill\ds -\frac{k^p}{r^{kp}}\int_\epsilon^r s^{kp-1}\nu(dd^cT,\varphi_\epsilon,s)ds\\
            & = &\ds k^p\left(\nu(T,\varphi_\epsilon,r)+\int_\epsilon^r \frac{\nu(dd^cT,\varphi_\epsilon,t)}t \left(\frac{t^p}{r^p} -\frac{t^{kp}}{r^{kp}}\right)dt \right).
          \end{array}
        \end{equation}
        Here we have used successively Equality (\ref{eq 1.4}) with $\psi_\epsilon$ instead of $\varphi_\epsilon$ in the first equality, then the change of variable $t=s^k$, next the fact that $\nu(dd^cT,\psi_\epsilon,s^k)=k^{p-1}\nu(dd^cT,\varphi_\epsilon,s)$ (equality proved  by Demailly in the case of closed positive currents), and finally Equality (\ref{eq 1.4}).\\
        When $\epsilon\to0$, Equality (\ref{eq 1.5}) implies Equality (\ref{eq 1.3}).\\
        Thanks to Equality (\ref{eq 1.3}), for every $r\in{}]0,R(\varphi)[$, we have
        $$\nu(T,\varphi,r)-\frac1{k^p}\nu(T,\varphi^k,r^k)=-\int_0^r\nu(dd^cT,\varphi,t) \left(\frac{t^{p-1}}{r^p} -\frac{t^{kp-1}}{r^{kp}}\right)dt.$$
        The current $dd^cT$ is positive and closed, hence $\nu(dd^cT,\varphi,.)$ is a non-negative increasing function and $\nu(dd^cT,\varphi)=0$. We consider two disjoint cases:
        \begin{itemize}
          \item \textit{First case $k\geq 1$.} We have
            $$\begin{array}{lcl}
                0 & \leq &\ds \int_0^r\nu(dd^cT,\varphi,t) \left(\frac{t^{p-1}}{r^p} -\frac{t^{kp-1}}{r^{kp}}\right)dt\\
                & \leq  &\ds  \nu(dd^cT,\varphi,r)\int_0^r \left(\frac{t^{p-1}}{r^p} -\frac{t^{kp-1}}{r^{kp}}\right)dt\\
                & = &\ds \nu(dd^cT,\varphi,r)\frac{k-1}{kp}.
              \end{array}$$
             So
             $$-\frac{k-1}{kp}\nu(dd^cT,\varphi,r)\leq\nu(T,\varphi,r)-\frac1{k^p}\nu(T,\varphi^k,r^k)\leq 0$$
          \item \textit{Second case $k< 1$.} A similar calculation shows that we have
            $$0\leq\nu(T,\varphi,r)-\frac1{k^p}\nu(T,\varphi^k,r^k)\leq -\frac{k-1}{kp}\nu(dd^cT,\varphi,r).$$
        \end{itemize}
        In the two cases, both terms (of the right-hand and the left-hand) have the same limit 0 when $r\to0^+$. This completes the proof of the proposition.
    \end{proof}
\section{Applications}
    The following corollaries are immediate consequences of Theorem \ref{theo1} and/or Proposition \ref{pro1}.
    \begin{cor}
        Let $T$ be a negative plurisubharmonic current of bidimension $(p,p)$ on $\Omega$ and $\varphi\in \PSH(T,\Omega)$. If there  exists $k\in\mathcal I_T(\varphi)$ such that the  function $t\longmapsto \frac{\nu(dd^cT,\varphi,t^{\frac1k})}{t}$ is integrable on a neighborhood of $0$, then the Lelong-Demailly number $\nu(T,\varphi)$ of $T$ relatively  to $\varphi$ exists.
    \end{cor}
    \begin{proof}
        We can prove this corollary in two different ways:
        \begin{itemize}
          \item \textit{First way}. Thanks to Proposition \ref{pro1}, to prove that $\nu(T,\varphi)$ exists, it suffices to show that $\nu(T,\varphi^k)$ exists; for this we observe that $$\frac{\nu(dd^cT,\varphi^k,t)}t=\frac{\nu(dd^cT,\varphi^k,(t^{\frac1k})^k)}t =k^{p-1}\frac{\nu(dd^cT,\varphi,t^{\frac1k})}t$$ is integrable on a neighborhood of 0 (hypothesis). Hence, thanks to Theorem  \ref{theo1}, $\nu(T,\varphi^k)$ exists.
          \item \textit{Second way}. We remark that the integrability of $\frac{\nu(dd^cT,\varphi,t^{\frac1k})}t$ is equivalent to the integrability of $\frac{\nu(dd^cT,\varphi,t)}t$. Hence, thanks to Theorem \ref{theo1}, $\nu(T,\varphi)$ exists. In fact if we take $t=s^k$ we obtain
              $$\int_0^{r_0}\frac{\nu(dd^cT,\varphi,t^{\frac1k})}t dt=k\int_0^{r_0^k}\frac{\nu(dd^cT,\varphi,s)}s ds.$$
        \end{itemize}
    \end{proof}
    \begin{cor}
        Let $T$ be a negative plurisubharmonic current of bidimension $(p,p)$ on $\Omega$ and $\varphi\in \PSH(T,\Omega)$ such that $\mathcal I_T(\varphi)$ is non bounded $($this is always the case if $\varphi$ is $\mathcal C^2)$. Then for every $r\in{}]0,R(\varphi)[$ one has
        $$\nu(T,\varphi,r)\leq -\int_0^r\nu(dd^cT,\varphi,t) \frac{t^{p-1}}{r^p}dt.$$
        In particular,
        $$\nu(T,\varphi,r)\leq -\frac{1-s^{-p}}p\nu(dd^cT,s\varphi,r)\quad\forall s\geq1.$$
    \end{cor}
    \begin{proof}
        Thanks to Proposition \ref{pro1}, for every $k\in\mathcal I_T(\varphi)$, we have
        $$u(k):=\frac1{k^p}\nu(T,\varphi^k,r^k)=\nu(T,\varphi,r)+\int_0^r\frac{\nu(dd^cT,\varphi,t)}t \left(\frac{t^p}{r^p} -\frac{t^{kp}}{r^{kp}}\right)dt.$$
        The function $u$ is non-positive and increasing on $\mathcal I_T(\varphi)$, so
        $$\lim_{k\to+\infty, k\in\mathcal I_T(\varphi)}u(k)=\nu(T,\varphi,r)+\int_0^r\nu(dd^cT,\varphi,t) \frac{t^{p-1}}{r^p}dt\leq 0.$$
        If $s=1$ the inequality is clear. For $s>1$ we have
        $$\begin{array}{lcl}
            \ds\nu(T,\varphi,r)\leq -\int_0^r\nu(dd^cT,\varphi,t) \frac{t^{p-1}}{r^p}dt 
            & \leq & \ds-\int_{r/s}^r\nu(dd^cT,\varphi,t) \frac{t^{p-1}}{r^p}dt \\
            & \leq &\ds -\frac{1-s^{-p}}p\nu(dd^cT,\varphi,\frac rs). 
          \end{array}
          $$
    \end{proof}
    \begin{theo}\label{theo2}
        Let $T$ be a positive or negative plurisubharmonic current of bidimension $(p,p)$ on $\Omega$ and $\varphi,\ \psi\in \PSH(T,\Omega)$ such that
        $$\liminf_{\varphi(z)\to0}\frac{\log \psi(z)}{\log \varphi(z)}\geq \ell$$
        where $\mathcal I_T(\varphi)$ contains a neighborhood $\mathscr V(\ell)$ of $\ell$. One has
        \begin{itemize}
          \item $\nu(T,\psi)\geq\ell^p\nu(T,\varphi)$ if $T$ is positive.
          \item $\nu(T,\psi)\leq\ell^p\nu(T,\varphi)$ if $T$ is negative $($with the assumption that $\psi$ satisfies condition $(C))$.
        \end{itemize}
        In particular,if $\log \psi(z)\sim\ell\log \varphi(z)$ for $\varphi(z)\in\mathscr V(0)$ then $\nu(T,\psi)=\ell^p\nu(T,\varphi).$
    \end{theo}
    This type of theorem is called a ``comparison theorem''; such a statement has been proved by Demailly in the case of positive closed currents. Furthermore, one can deduce the invariance of  the Lelong-Demailly numbers under changes of coordinate systems.
    \begin{proof}
        One can replace $\psi$ by $\psi^k$ with $1<k\in\mathcal I_T(\psi)$, in order to assume that $$\liminf_{\varphi(z)\to0}\frac{\log \psi(z)}{\log \varphi(z)}> \ell.$$
        As a consequence $$\lim_{\varphi(z)\to0}\frac{\psi(z)}{\varphi(z)^\ell}=0.$$
        For $\epsilon>0$ small sufficiently, we set $\Psi_\epsilon:=\psi+\epsilon\varphi^\ell \underset{\varphi(z)\in\mathscr V(0)}\sim\epsilon\varphi^\ell$. Thanks to the dominated convergence theorem,
        $$\lim_{\epsilon\to0}\frac1{r^p}\int_{\Psi_\epsilon<r}T\w\beta_{\Psi_\epsilon}^p =\frac1{r^p}\int_{\psi<r}T\w\beta_{\psi}^p.$$
        \begin{itemize}
          \item \textit{First assume $T\geq 0$.} As the function $\nu(T,\Psi_\epsilon,.)$ is increasing then
            $$\begin{array}{lcl}
                \nu(T,\Psi_\epsilon,r)&=&\ds\frac1{r^p}\int_{\Psi_\epsilon<r}T\w\beta_{\Psi_\epsilon}^p\\
                & \geq &\ds \lim_{\rho\to0}\frac1{\rho^p}\int_{\Psi_\epsilon<\rho}T\w\beta_{\Psi_\epsilon}^p =\lim_{\rho\to0}\frac1{\rho^p}\int_{\epsilon\varphi^\ell<\rho}T\w\beta_{\Psi_\epsilon}^p \\
                & \geq &\ds\lim_{\rho\to0}\frac1{\rho^p}\int_{\epsilon\varphi^\ell<\rho}T\w\beta_{\epsilon\varphi^\ell}^p =\nu(T,\varphi^\ell).
              \end{array}$$

            because $\Psi_\epsilon\sim\epsilon\varphi^\ell$ and $T\w(dd^c(\psi+\epsilon\varphi^\ell))^p\geq T\w(\epsilon dd^c(\varphi^\ell))^p$. If $\epsilon\to0$, we obtain $\nu(T,\psi,r)\geq\nu(T,\varphi^\ell)$. Hence, if $r\to0$,  $\nu(T,\psi)\geq\nu(T,\varphi^\ell)$.
          \item \textit{Now, assume $T\leq0$.} Thanks to Theorem \ref{theo1}, Equality (\ref{eq 1.2}) gives
            $$\begin{array}{lcl}
                \ds\nu(T,\Psi_\epsilon,r)+\int_0^r\frac{\nu(dd^cT,\Psi_\epsilon,t)}t\left(\frac{t^p}{r^p}-1\right)dt & \leq &\ds \nu(T,\Psi_\epsilon) \\
                & = &\ds  \lim_{\rho\to0}\frac1{\rho^p}\int_{\Psi_\epsilon<\rho}T\w\beta_{\Psi_\epsilon}^p
              \end{array}$$
            and a similar calculation proves that $$\nu(T,\Psi_\epsilon,r)+\int_0^r\frac{\nu(dd^cT,\Psi_\epsilon,t)}t\left(\frac{t^p}{r^p}-1\right)dt \leq \nu(T,\varphi^\ell).$$
            So if $\epsilon\to0$ we obtain
            $$\nu(T,\psi,r)+\int_0^r\frac{\nu(dd^cT,\psi,t)}t\left(\frac{t^p}{r^p}-1\right)dt \leq \nu(T,\varphi^\ell).$$
            When $r\to0$, the last inequality gives $\nu(T,\psi)\leq \nu(T,\varphi^\ell).$
        \end{itemize}
        In the particular case $\log \psi(z)\sim\ell\log \varphi(z)$, for every $\epsilon>0$, there exists $\eta>0$ so that if $|\varphi(z)|<\eta$, then $\varphi(z)^{\ell(1+\epsilon)}\leq \psi(z)\leq \varphi(z)^{\ell(1-\epsilon)}$ and for $\epsilon$ small enough, $\ell(1+\epsilon),\ \ell(1-\epsilon)\in\mathscr V(\ell)\subset\mathcal I_T(\varphi)$. Therefore we can apply the previous inequalities to obtain $\ell^p(1+\epsilon)^p\nu(T,\varphi)\leq\nu(T,\psi)\leq \ell^p(1-\epsilon)^p\nu(T,\varphi)$ (resp. $\ell^p(1-\epsilon)^p\nu(T,\varphi)\leq\nu(T,\psi)\leq \ell^p(1+\epsilon)^p\nu(T,\varphi)$ in the case $T\geq0$ (resp. $T\leq0$). Hence, when $\epsilon\to0$, we obtain $\nu(T,\psi)=\ell^p\nu(T,\varphi)$.
    \end{proof}
\section*{Acknowledgements}
     I thank Professor Jean-Pierre Demailly for several suggestions which enabled me to clarify and improve the original form of this note. I thank also Professor Khalifa Dabbek for useful discussions concerning this work.

\end{document}